\documentclass{amsart}

\usepackage[T1]{fontenc}
\usepackage[utf8]{inputenc}
\usepackage{amsmath,amssymb,amsthm,mathtools}

\usepackage{enumitem}
\usepackage{microtype}

\usepackage[
  colorlinks=true,
  linkcolor=blue,
  citecolor=blue,
  urlcolor=blue
]{hyperref}
\usepackage{doi}

\newtheorem{theorem}{Theorem}[section]
\newtheorem{lemma}[theorem]{Lemma}
\newtheorem{proposition}[theorem]{Proposition}
\newtheorem{corollary}[theorem]{Corollary}

\theoremstyle{definition}

\newtheorem{example}[theorem]{Example}

\theoremstyle{remark}

\DeclareMathOperator{\Aut}{Aut}
\DeclareMathOperator{\ord}{ord}
\DeclareMathOperator{\Tor}{Tor}

\newcommand{\id}{\operatorname{id}}
\newcommand{\Inv}{\iota}
\newcommand{\Nset}{\mathcal N(G)}
\newcommand{\Pset}{\mathcal P(G)}
\newcommand{\Z}{\mathbb Z}

\title[Inverse Ambiguous Maps on Infinite Groups]
{Inverse Ambiguous Maps on Infinite Groups}

\author{Sezen Bostan}
\address{%
T\"UB\.{I}TAK ULAKB\.{I}M\\
Middle East Technical University\\
Modsimmer Building\\
2nd Floor, Room 203\\
06800 \c{C}ankaya, Ankara\\
Turkey
}
\email{sezen.bostan@tubitak.gov.tr}

\author{K{\i}van\c{c} Ersoy}
\address{%
Freie Universit\"at Berlin\\
Fachbereich Mathematik und Informatik\\
Arnimallee 7\\
14195 Berlin\\
Germany
}
\email{ersoy@zedat.fu-berlin.de}

\dedicatory{%
Dedicated to Professor Mahmut Kuzucuo\u{g}lu
on the occasion of his retirement.
}

\subjclass{%
Primary 20E36;
Secondary 20K30, 20E26, 20F50, 39B12}

\keywords{%
inverse ambiguous function,
inverse ambiguous automorphism,
functional equation,
finitely generated abelian group,
divisible abelian group,
locally finite group,
residually finite group}

\begin{document}
\begin{abstract}
Let $G$ be a group. A bijection $f\colon G\to G$ is called inverse
ambiguous if $f^{-1}(x)=f(x)^{-1}$ for every $x\in G$. We prove that
every infinite group admits an inverse ambiguous function. An inverse
ambiguous automorphism can exist only on an abelian group. We classify
the finitely generated and divisible abelian groups admitting such
automorphisms. If $A\cong\Z^r\oplus T$, where $T$ is finite, then $A$
admits an inverse ambiguous automorphism if and only if $r$ is even and
$T$ admits one. Together with Toborg's finite classification, this
gives an explicit classification of the finitely generated abelian
groups admitting such automorphisms. For a divisible abelian group
$A\cong\mathbb Q^{(\kappa_0)}\oplus
\bigoplus_p C_{p^\infty}^{(\kappa_p)}$, such an automorphism exists if
and only if none of $\kappa_0$, $\kappa_2$, or $\kappa_p$ for
$p\equiv3\pmod4$ is a finite odd cardinal. We also prove that if
every proper subgroup of an infinite locally finite group $G$ admits an
inverse ambiguous
automorphism, then $G$ admits one that leaves every subgroup invariant.
Finally, we prove that a residually finite group whose finite
quotients admit inverse ambiguous automorphisms is abelian. However, we  prove the existence of infinite
residually finite abelian groups whose finite quotients all admit such
automorphisms, but the group itself does not.
\end{abstract}
\maketitle

\section{Introduction}

Let $G$ be a group and let $f\colon G\to G$ be a bijection. The
expressions $f(x)^{-1}$ and $f^{-1}(x)$ have different meanings: the
first is the inverse of $f(x)$ in $G$, while the second is the
preimage of $x$ under $f$. The map $f$ is called \emph{inverse
ambiguous} if
\[
f^{-1}(x)=f(x)^{-1}
\qquad (x\in G).
\]
If
\[
\iota\colon G\to G,
\qquad
\iota(x)=x^{-1},
\]
denotes the inversion permutation, then this condition is equivalent
to
\[
f^2=\iota.
\]
Thus inverse ambiguous functions are exactly the square roots of the
inversion permutation.

Related functional equations were first studied in analysis. Euler
and Foran~\cite{EulerForan} considered functions whose inverse is
their reciprocal. Cheng, Dasgupta, Ebanks, Kinch, Larson, and
McFadden~\cite{ChengEtAl} studied the equation $f^{-1}=1/f$, and
Griffiths~\cite{Griffiths} considered equations involving second
iterates. Schmitz introduced inverse ambiguous functions in the
group-theoretic setting~\cite{Schmitz}. He proved that a finite group
admits such a function if and only if the number of its elements of
order at least three is divisible by four. Schmitz and
Gallagher~\cite{SchmitzGallagher} applied this criterion to several
families of finite non-abelian groups. Toborg obtained further
criteria for finite groups and classified the finite groups admitting
inverse ambiguous automorphisms~\cite{Toborg}. Continuous inverse
ambiguous functions on Lie groups were studied by Schmitz, Rahman,
and Kindness~\cite{schmitzrahmankindness2025continuous}.

In this paper, we study these questions for infinite groups in the
algebraic setting. We first prove that every
infinite group admits an inverse ambiguous function. Together with
the finite criterion, this shows that a group admits such a function
if and only if it is infinite or is finite and the number of its
elements of order at least three is divisible by four.

The automorphism case is more rigid. Toborg's
argument in the finite case applies without change to arbitrary
groups and shows that every group admitting an inverse ambiguous
automorphism is abelian. For an abelian group $A$, written
additively, the condition is equivalent to the existence of an
automorphism $J$ such that
\[
J^2=-\id_A.
\]
Equivalently, the underlying abelian group admits a
$\mathbb Z[i]$-module structure. We use this observation to classify
finitely generated and divisible abelian groups admitting inverse
ambiguous automorphisms. If
\[
A\cong\mathbb Z^r\oplus T,
\]
where $T$ is finite, then $A$ admits such an automorphism if and only
if $r$ is even and $T$ admits one. Combining this with Toborg's finite
classification gives an explicit description in terms of the primary
components of $T$.

If $A$ is divisible, write
\[
A\cong\mathbb Q^{(\kappa_0)}\oplus
\bigoplus_p C_{p^\infty}^{(\kappa_p)},
\]
where $C_{p^\infty}$ denotes the Pr\"ufer $p$-group and
$\kappa_0,\kappa_p$ are cardinals. We prove that $A$ admits an inverse
ambiguous automorphism if and only if none of $\kappa_0$, $\kappa_2$,
or $\kappa_p$ for $p\equiv3\pmod4$ is a finite odd cardinal. No
restriction is imposed on $\kappa_p$ for $p\equiv1\pmod4$.

We also study groups whose proper subgroups admit inverse ambiguous
automorphisms. We prove that if $G$ is an infinite locally finite
group and every proper subgroup of $G$ admits such an automorphism,
then $G$ itself admits an inverse ambiguous automorphism that
stabilizes every subgroup of $G$ setwise. A finite semidirect product
shows that infiniteness is necessary, while a Tarski monster of
suitable prime exponent~\cite{OlshanskiiTarski} shows that local
finiteness cannot be omitted.

Finally, we consider finite quotients of residually finite groups. If
every finite quotient of a residually finite group $G$ admits an
inverse ambiguous automorphism, then $G$ is abelian. The existence of
such an automorphism does not always pass from the finite quotients to
the group. We construct an infinite residually finite abelian group
whose finite quotients all admit inverse ambiguous automorphisms, but
the group itself admits none.

Section~2 gives the square-root formulation and recalls the finite
criterion. Section~3 is about infinite groups and inverse ambiguous
automorphisms. Section~4 deals with finitely generated and divisible
abelian groups. Section~5 contains the local-to-global theorem and the
results on finite quotients.

\section{Preliminaries}

Throughout the paper, groups are written multiplicatively unless an
abelian group is explicitly written additively.

Let
\[
\Pset=\bigl\{\{x,x^{-1}\}:x\in G,\ x^2\neq1\bigr\}
\]
denote the set of two-element orbits of inversion.

\begin{proposition}\label{prop:pairing}
A group $G$ admits an inverse ambiguous function if and only if
$\Pset$ can be partitioned into two-element subsets.
\end{proposition}

\begin{proof}
Suppose that $f^2=\Inv$. Since $f$ commutes with $\Inv$, it induces a
permutation $\widehat f$ of $\Pset$. Since inversion fixes each member
of $\Pset$ setwise, $\widehat f^{\,2}=\id$. Moreover, $\widehat f$ has
no fixed point: if $\{x,x^{-1}\}$ were fixed, then $f$ would restrict
to a permutation of a two-element set whose square is the non-trivial
transposition, which is impossible. Hence all orbits of $\widehat f$ have size two.

Conversely, pair the members of $\Pset$. For each pair
\[
\bigl\{\{a,a^{-1}\},\{b,b^{-1}\}\bigr\},
\]
define
\[
a\mapsto b\mapsto a^{-1}\mapsto b^{-1}\mapsto a,
\]
and fix every $x$ satisfying $x^2=1$. The resulting permutation $f$
satisfies $f^2=\Inv$, and is therefore inverse ambiguous.
\end{proof}

The finite case is known.

\begin{theorem}[Schmitz]
\label{thm:finite}
Let $G$ be a finite group. The following are equivalent.
\begin{enumerate}[label=\textup{(\roman*)}]
\item $G$ admits an inverse ambiguous function;
\item the number of elements of order at least three is divisible by
four;
\item the set $\Pset$ has even cardinality.
\end{enumerate}
\end{theorem}

\begin{proof}
The equivalence of (i) and (ii) is due to Schmitz~\cite{Schmitz}. The
number of elements of order at least three is $2|\Pset|$, so (ii) and
(iii) are equivalent. For further equivalent conditions in terms of
involutions and Sylow $2$-subgroups, see
Toborg~\cite[Lemma~2.2 and Theorem~2.3]{Toborg}.
\end{proof}

\section{Infinite Groups}
\subsection{Existence}
Let
\[
\Nset=\{x\in G:x^2\neq1\}
\]
be the set of elements of $G$ with non-trivial square.

\begin{lemma}\label{lem:noninv}
Let $G$ be an infinite group. Then $\Nset$ is either empty or infinite.
\end{lemma}

\begin{proof}
Suppose that $\Nset$ is finite and non-empty, and set
\[
I=\{t\in G:t^2=1\}.
\]
Then $I$ is infinite. Choose $a\in\Nset$. Since the map $t\mapsto at$
is injective and $\Nset$ is finite, only finitely many $t\in I$
satisfy $at\in\Nset$. Hence
\[
T=\{t\in I:(at)^2=1\}
\]
is infinite.

For every $t\in T$, we have $tat=a^{-1}$. Fix $t_0\in T$ and set
$c_t=t_0t$. Then the elements $c_t$ are distinct and satisfy
\[
c_tac_t^{-1}=a,
\]
so they centralize $a$.

Since $\Nset$ is finite, infinitely many $c_t$ satisfy $c_t^2=1$.
For each such $c_t$,
\[
(ac_t)^2=a^2\neq1,
\]
so $ac_t\in\Nset$. These elements are distinct, giving infinitely many
elements of $\Nset$, a contradiction.
\end{proof}

\begin{theorem}\label{thm:infinite}
Every infinite group admits an inverse ambiguous function.
\end{theorem}

\begin{proof}
If $\Nset=\varnothing$, then inversion is the identity, so the identity
map is inverse ambiguous. Otherwise $\Nset$, and hence $\Pset$, is
infinite by Lemma~\ref{lem:noninv}. Pair the members of $\Pset$ and
apply Proposition~\ref{prop:pairing}.
\end{proof}

Combining this theorem with the finite criterion gives the following
complete existence criterion.

\begin{corollary}\label{cor:all-groups}
A group $G$ admits an inverse ambiguous function if and only if either
$G$ is infinite, or $G$ is finite and
\[
|\{x\in G:\ord(x)\geq3\}|\equiv0\pmod4.
\]
\end{corollary}

\subsection{Automorphisms}

Toborg's proof~\cite[Theorem~3.2]{Toborg} does not use finiteness and
therefore applies to arbitrary groups.

\begin{proposition}\label{prop:automorphism-rigidity}
Let $G$ be a group and let $\alpha\in\Aut(G)$ be inverse ambiguous.
Then $G$ is abelian and $\alpha^4=\id$. Moreover, exactly one of the
following holds:
\begin{enumerate}[label=\textup{(\roman*)}]
\item $\alpha$ has order $4$;
\item every element of $G$ has order dividing $2$ and
$\alpha^2=\id$.
\end{enumerate}
\end{proposition}

\begin{proof}
By the square-root formulation, $\alpha^2(x)=x^{-1}$ for every
$x\in G$. Since $\alpha^2$ is a homomorphism, for all $x,y\in G$ we
have
\[
x^{-1}y^{-1}=\alpha^2(xy)=(xy)^{-1}=y^{-1}x^{-1}.
\]
Thus $x$ and $y$ commute, and hence $G$ is abelian. Also,
$\alpha^4=\Inv^2=\id$. If every element has order dividing $2$, then
$\Inv=\id$ and $\alpha^2=\id$. Otherwise
$\alpha^2=\Inv\neq\id$, so $\alpha$ has order $4$.
\end{proof}

For abelian groups, the automorphism equation has a module-theoretic
interpretation.

\begin{proposition}\label{prop:gaussian}
Let $A$ be an abelian group, written additively. The following are
equivalent.
\begin{enumerate}[label=\textup{(\roman*)}]
\item $A$ admits an inverse ambiguous automorphism;
\item there exists $J\in\Aut(A)$ such that $J^2=-\id$;
\item the underlying abelian group $A$ admits a $\Z[i]$-module
structure.
\end{enumerate}
\end{proposition}

\begin{proof}
The equivalence of (i) and (ii) follows from the square-root
formulation above, written additively. If (ii) holds, define
\[
(m+ni)\cdot a=ma+nJ(a)
\]
for $m,n\in\Z$ and $a\in A$. The relation $J^2=-\id$ gives the
defining relation $i^2=-1$, so this is a $\Z[i]$-module structure.
Conversely, multiplication by $i$ on a $\Z[i]$-module is an
automorphism whose square is $-\id$.
\end{proof}

\section{Abelian Groups}
\subsection{Finitely generated abelian groups}
We now combine Proposition~\ref{prop:gaussian} with the structure
theorem for finitely generated abelian groups.
\begin{theorem}\label{thm:fga}
Let $A$ be a finitely generated abelian group, let $T=\Tor(A)$, and
write
\[
A\cong\Z^r\oplus T.
\]
Then $A$ admits an inverse ambiguous automorphism if and only if $r$ is
even and $T$ admits an inverse ambiguous automorphism.
\end{theorem}

\begin{proof}
Suppose that $J\in\Aut(A)$ satisfies $J^2=-\id_A$. Since $T$ is
characteristic in $A$, the restriction $J|_T$ satisfies
$(J|_T)^2=-\id_T$. Moreover, $J$ induces an automorphism $\overline{J}$ of
$A/T\cong\Z^r$ satisfying
\[
\overline{J}^{\,2}=-\id_{A/T}.
\]

Let $M\in\operatorname{GL}_r(\Z)$ be the matrix of $\overline{J}$.
Then $M^2=-I_r$. If $r$ were odd, taking determinants would give
$\det(M)^2=\det(-I_r)=-1$, which is impossible in $\Z$. Hence $r$ is even.

Conversely, suppose that $r=2m$ and that
$J_T\in\Aut(T)$ satisfies
\[
J_T^2=-\id_T.
\]
Define $J_0\in\Aut(\Z^{2m})$ by
\[
J_0(x_1,x_2,\ldots,x_{2m-1},x_{2m})
=
(-x_2,x_1,\ldots,-x_{2m},x_{2m-1}).
\]
Then $J_0^2=-\id_{\Z^{2m}}$. Therefore
$J_0\oplus J_T$ is an automorphism of $\Z^{2m}\oplus T$
whose square is $-\id$.
Hence $A$ admits an inverse ambiguous automorphism.
\end{proof}

\begin{corollary}\label{cor:free}
The free abelian group $\Z^r$ admits an inverse ambiguous automorphism
if and only if $r$ is even.
\end{corollary}

Toborg's classification of finite abelian groups gives an explicit
form of Theorem~\ref{thm:fga}.

\begin{corollary}\label{cor:explicit-fga}
Let $A$ be an infinite finitely generated abelian group. Then $A$
admits an inverse ambiguous automorphism if and only if
\[
A\cong\Z^{2m}\oplus T
\]
for some $m\geq1$, where the primary components of the finite abelian
group $T$ satisfy the following conditions.
\begin{enumerate}[label=\textup{(\roman*)}]
\item If $p\equiv1\pmod4$, then $T_p$ is arbitrary.
\item If $p\equiv3\pmod4$, then
\[
T_p\cong\bigoplus_{k\geq1}C_{p^k}^{\,2m_{p,k}}
\]
for non-negative integers $m_{p,k}$, all but finitely many of which are
zero.
\item The $2$-primary component has the form
\[
T_2\cong E\oplus
\bigoplus_{j=1}^s
\left(C_{2^{a_j}}\oplus C_{2^{b_j}}\right),
\]
where $s$ is a non-negative integer, $E$ is an elementary abelian
$2$-group, the integers $a_j,b_j$ are positive, and
$|a_j-b_j|\leq1$ for every $j$.
\end{enumerate}
\end{corollary}

\begin{proof}
By Theorem~\ref{thm:fga}, it remains to determine when $T$ admits an
inverse ambiguous automorphism. The stated conditions are exactly those in
Toborg's classification of the odd-primary and $2$-primary components;
see
\cite[Theorems~3.7, 3.11, 4.5, and~4.6]{Toborg}.
\end{proof}

\subsection{Divisible abelian groups}
We next classify the divisible abelian groups admitting inverse
ambiguous automorphisms.

\begin{theorem}\label{thm:divisible_abelian}
Let $G$ be a divisible abelian group, written additively, and write
\[
G\cong
\mathbb Q^{(\kappa_0)}
\oplus
\bigoplus_p C_{p^\infty}^{(\kappa_p)},
\]
where $p$ ranges over the primes, $C_{p^\infty}$ denotes the Pr\"ufer
$p$-group, and $\kappa_0,\kappa_p$ are cardinals. Then $G$ admits an
inverse ambiguous automorphism if and only if none of the cardinals
\[
\kappa_0,\qquad \kappa_2,\qquad
\kappa_p \quad (p\equiv3\pmod4)
\]
is a finite odd cardinal. There is no restriction on $\kappa_p$ for
primes $p\equiv1\pmod4$.
\end{theorem}

\begin{proof}
By Proposition~\ref{prop:gaussian}, it is enough to determine when
there exists $J\in\Aut(G)$ such that
\[
J^2=-\id_G.
\]

Suppose first that such an automorphism $J$ exists, and let
\[
T=\Tor(G).
\]
Since $T$ is characteristic in $G$, the automorphism $J$ restricts to
an automorphism of $T$ and induces an automorphism $\overline J$ of
\[
G/T\cong\mathbb Q^{(\kappa_0)}
\]
satisfying
\[
\overline J^{\,2}=-\id_{G/T}.
\]

Suppose that $\kappa_0=n$ is finite. Every additive endomorphism of the
$\mathbb Q$-vector space $\mathbb Q^n$ is $\mathbb Q$-linear. Thus
$\overline J$ is represented by a matrix
$M\in\operatorname{GL}_n(\mathbb Q)$ satisfying
\[
M^2=-I_n.
\]
Taking determinants gives
\[
\det(M)^2=(-1)^n.
\]
If $n$ were odd, this would make $-1$ a square in $\mathbb Q$, which
is impossible. Hence $\kappa_0$ is even whenever it is finite.

For each prime $p$, let $G_p$ denote the $p$-primary component of $T$.
Then
\[
G_p\cong C_{p^\infty}^{(\kappa_p)},
\]
and $G_p$ is characteristic in $G$, so it is $J$-invariant.

Let $p\equiv3\pmod4$, and suppose that $\kappa_p=n$ is finite. The
subgroup
\[
G_p[p]=\{x\in G_p:px=0\}
\]
is $J$-invariant and is isomorphic to the $n$-dimensional vector space
$\mathbb F_p^n$. The restriction of $J$ to $G_p[p]$ is therefore
represented by a matrix $M\in\operatorname{GL}_n(\mathbb F_p)$
satisfying
\[
M^2=-I_n.
\]
If $n$ were odd, then
\[
\det(M)^2=-1
\]
in $\mathbb F_p$. This is impossible because $-1$ is not a quadratic
residue modulo $p$ when $p\equiv3\pmod4$. Thus $\kappa_p$ is even
whenever it is finite.

It remains to consider $p=2$. Suppose that $\kappa_2=n$ is finite. The
subgroup
\[
G_2[4]=\{x\in G_2:4x=0\}
\]
is $J$-invariant and is isomorphic to
\[
(\mathbb Z/4\mathbb Z)^n.
\]
Hence the restriction of $J$ is represented by some
\[
M\in\operatorname{GL}_n(\mathbb Z/4\mathbb Z)
\]
with $M^2=-I_n$. If $n$ were odd, taking determinants would give
\[
\det(M)^2=-1
\qquad\text{in }\mathbb Z/4\mathbb Z.
\]
However, every unit of $\mathbb Z/4\mathbb Z$ has square $1$. This
contradiction shows that $\kappa_2$ is even whenever it is finite.

Conversely, suppose that none of $\kappa_0$, $\kappa_2$, or
$\kappa_p$ for $p\equiv3\pmod4$ is a finite odd cardinal. If
$\kappa$ is an even finite cardinal or an infinite cardinal, then a set
of cardinality $\kappa$ can be partitioned into pairs. Consequently,
the direct summands can be paired in each of
\[
\mathbb Q^{(\kappa_0)},\qquad
C_{2^\infty}^{(\kappa_2)},\qquad
C_{p^\infty}^{(\kappa_p)}
\quad (p\equiv3\pmod4).
\]
On each pair $H\oplus H$, where $H=\mathbb Q$ or
$H=C_{p^\infty}$, define
\[
R_H(x,y)=(-y,x).
\]
Then
\[
R_H^2(x,y)=(-x,-y),
\]
so the direct sum of these maps has square $-\id$ on the corresponding
component.

It remains to consider primes $p\equiv1\pmod4$. Fix such a prime. We
construct a compatible sequence of integers $(u_{p,n})_{n\geq1}$
satisfying
\[
u_{p,n}^2\equiv-1\pmod{p^n}
\qquad\text{and}\qquad
u_{p,n+1}\equiv u_{p,n}\pmod{p^n}.
\]
Since $p\equiv1\pmod4$, there exists $u_{p,1}$ such that
\[
u_{p,1}^2\equiv-1\pmod p.
\]
Suppose that $u_{p,n}$ has been chosen and write
\[
u_{p,n}^2+1=cp^n.
\]
Since $p\nmid2u_{p,n}$, we can choose $t\in\mathbb Z$ such that
\[
c+2u_{p,n}t\equiv0\pmod p.
\]
Set
\[
u_{p,n+1}=u_{p,n}+tp^n.
\]
Then
\[
u_{p,n+1}^2+1
=
p^n\bigl(c+2u_{p,n}t+t^2p^n\bigr),
\]
which is divisible by $p^{n+1}$, and the required compatibility is
immediate.

For $x\in G_p=C_{p^\infty}^{(\kappa_p)}$, choose $n$ such that
$p^nx=0$, and define
\[
J_p(x)=u_{p,n}x.
\]
The compatibility of the sequence $(u_{p,n})_{n\geq1}$ shows that
this definition is independent of the choice of $n$. Given
$x,y\in G_p$, choose $n$ such that $p^nx=p^ny=0$. Then
\[
J_p(x+y)=u_{p,n}(x+y)=J_p(x)+J_p(y),
\]
so $J_p$ is an endomorphism of $G_p$. Moreover,
\[
J_p^2(x)=u_{p,n}^2x=-x.
\]
Thus $J_p^2=-\id_{G_p}$, and hence $J_p$ is an automorphism.

Taking the direct sum of the maps constructed on the rational
component and all the primary components gives an automorphism
$J\in\Aut(G)$ such that
\[
J^2=-\id_G.
\]
Therefore $J$ is an inverse ambiguous automorphism.
\end{proof}

\section{Local and Finite-Quotient Results}

\subsection{Proper subgroups of locally finite groups}

We use the following cyclic-group criterion of
Toborg~\cite[Lemma~3.6]{Toborg}.

\begin{lemma}\label{lem:cyclic}
Let $C$ be a non-trivial finite cyclic $p$-group. Then $C$ admits an
inverse ambiguous automorphism if and only if either
\[
p\equiv 1\pmod 4
\]
or
\[
C\cong C_2.
\]
\end{lemma}

The next lemma extends the part of Toborg's finite result corresponding
to primes congruent to $1$ modulo $4$ to arbitrary torsion abelian
groups; compare \cite[Theorem~3.7]{Toborg}.

\begin{lemma}\label{lem:torsion-primary}
Let $A$ be a torsion abelian group, and let $A_p$ denote its
$p$-primary component. Suppose that $A_2$ is elementary abelian and
that
\[
A_p=0
\]
for every prime $p\equiv 3\pmod 4$. Then there exists an automorphism
$J\in\Aut(A)$ such that
\[
J^2=-\id_A
\]
and
\[
J(H)=H
\]
for every subgroup $H\leq A$. In particular, $J$ is an inverse
ambiguous automorphism.
\end{lemma}

\begin{proof}
By the primary decomposition of torsion abelian groups,
\[
A=A_2\oplus\bigoplus_{p\equiv 1\pmod 4}A_p.
\]

Fix a prime $p\equiv1\pmod4$. As in the proof of
Theorem~\ref{thm:divisible_abelian}, choose integers
$(u_{p,n})_{n\geq1}$ such that
\begin{equation}\label{eq:compatible-roots}
u_{p,n}^2\equiv-1\pmod{p^n}
\quad\text{and}\quad
u_{p,n+1}\equiv u_{p,n}\pmod{p^n}.
\end{equation}

For $x\in A_p$, choose $n\geq 1$ such that $p^n x=0$, and define
\[
J_p(x)=u_{p,n}x.
\]
This definition is independent of the choice of $n$. Indeed, suppose
also that $p^m x=0$, and assume without loss of generality that
$m\geq n$. The compatibility congruences imply that
\[
u_{p,m}\equiv u_{p,n}\pmod{p^n}.
\]
Consequently,
\[
u_{p,m}x=u_{p,n}x.
\]

The map $J_p$ is an endomorphism of $A_p$. To see this, let
$x,y\in A_p$, and choose $n$ such that
\[
p^n x=p^n y=0.
\]
Then
\[
J_p(x+y)
 =u_{p,n}(x+y)
 =u_{p,n}x+u_{p,n}y
 =J_p(x)+J_p(y).
\]
Moreover,
\[
J_p^2(x)
 =u_{p,n}^2x
 =-x,
\]
because $p^n$ divides $u_{p,n}^2+1$. Hence
\[
J_p^2=-\id_{A_p}.
\]

On $A_2$, define
\[
J_2=\id_{A_2}.
\]
Since $A_2$ is elementary abelian, $-x=x$ for every $x\in A_2$.
Therefore
\[
J_2^2=-\id_{A_2}.
\]

Now define
\[
J
 =
J_2\oplus\bigoplus_{p\equiv 1\pmod 4}J_p.
\]
Then
\[
J^2=-\id_A,
\]
so $J$ is an automorphism with
\[
J^{-1}=-J.
\]

It remains to prove that every subgroup of $A$ is $J$-invariant.
Let $x\in A$. Since $x$ has finite order, it has only finitely many
non-zero primary components. Write
\[
x=x_2+x_{p_1}+\cdots+x_{p_r},
\]
where the $p_i$ are distinct primes congruent to $1$ modulo $4$.
Choose positive integers $e_i$ such that
\[
p_i^{e_i}x_{p_i}=0
\qquad (1\leq i\leq r).
\]
By the Chinese remainder theorem, there exists an integer $m$ such
that
\[
m\equiv 1\pmod 2
\]
and
\[
m\equiv u_{p_i,e_i}\pmod{p_i^{e_i}}
\qquad (1\leq i\leq r).
\]
Since $A_2$ has exponent at most $2$, it follows that
\[
mx_2=x_2.
\]
For each $i$, we also have
\[
mx_{p_i}=u_{p_i,e_i}x_{p_i}=J_{p_i}(x_{p_i}).
\]
Therefore
\[
J(x)=mx\in\langle x\rangle.
\]

Let $H\leq A$. The preceding observation gives
\[
J(H)\subseteq H.
\]
Since $J^{-1}(x)=-J(x)\in\langle x\rangle$ for every $x\in A$,
we also have
\[
J^{-1}(H)\subseteq H.
\]
Applying $J$ to the latter inclusion yields
\[
H\subseteq J(H).
\]
Thus
\[
J(H)=H.
\]
\end{proof}

\begin{theorem}\label{thm:proper-subgroups}
Let $G$ be an infinite locally finite group. If every proper subgroup
of $G$ admits an inverse ambiguous automorphism, then $G$ admits an
inverse ambiguous automorphism that leaves every subgroup of $G$
invariant. In particular, $G$ is abelian.
\end{theorem}

\begin{proof}
Let $x,y\in G$. Since $G$ is locally finite, the subgroup
\[
\langle x,y\rangle
\]
is finite. Since $G$ is infinite, this subgroup is proper. By
hypothesis, it admits an inverse ambiguous automorphism, and hence it
is abelian by Proposition~\ref{prop:automorphism-rigidity}. Therefore
$x$ and $y$ commute. Since $x$ and $y$ were arbitrary, $G$ is
abelian.

We now write $G$ additively. Since $G$ is locally finite, it is a
torsion group. For each prime $p$, let $G_p$ denote the $p$-primary
component of $G$.

Let $0\neq x\in G_p$. The cyclic subgroup $\langle x\rangle$ is
finite, non-trivial, and proper in $G$. It therefore admits an inverse
ambiguous automorphism. By Lemma~\ref{lem:cyclic}, either
\[
p\equiv 1\pmod 4
\]
or
\[
\langle x\rangle\cong C_2.
\]
It follows that
\[
G_p=0
\]
for every prime $p\equiv 3\pmod 4$, and that every non-zero element
of $G_2$ has order $2$. Thus $G_2$ is elementary abelian.

The hypotheses of Lemma~\ref{lem:torsion-primary} are therefore
satisfied. Hence $G$ admits an inverse ambiguous automorphism that
leaves every subgroup of $G$ invariant.
\end{proof}

The following examples show that neither infiniteness nor local
finiteness can be omitted from Theorem~\ref{thm:proper-subgroups}.

\begin{example}\label{ex:finite-counterexample}
Let
\[
G=C_{41}\rtimes C_5
 =
\langle a,b\mid
a^{41}=b^5=1,\ b^{-1}ab=a^{10}
\rangle.
\]
Since
\[
10^5\equiv 1\pmod{41}
\quad\text{and}\quad
10\not\equiv 1\pmod{41},
\]
the residue class of $10$ has order $5$ in
$(\mathbb Z/41\mathbb Z)^\times$. Thus $G$ is a non-abelian group of
order $205$.

By Lagrange's theorem, every proper subgroup of $G$ has order $1$,
$5$, or $41$. Hence every non-trivial proper subgroup is cyclic of
prime order. Since
\[
5\equiv 41\equiv 1\pmod 4,
\]
Lemma~\ref{lem:cyclic} shows that every non-trivial proper subgroup
of $G$ admits an inverse ambiguous automorphism. The trivial subgroup
admits the identity automorphism.

On the other hand, $G$ itself admits no inverse ambiguous
automorphism, since it is non-abelian. Thus the infiniteness
assumption in Theorem~\ref{thm:proper-subgroups} cannot be omitted.
\end{example}

\begin{example}\label{ex:tarski-counterexample}
Let $p$ be a sufficiently large prime satisfying
\[
p\equiv 1\pmod 4,
\]
and let $G$ be a Tarski monster of exponent $p$ constructed by
Ol'shanskii~\cite{OlshanskiiTarski}. Every non-trivial proper subgroup
of $G$ is cyclic of order $p$, and hence admits an inverse ambiguous
automorphism by Lemma~\ref{lem:cyclic}. The trivial subgroup admits
the identity automorphism.

The group $G$ is non-abelian, so it admits no inverse ambiguous
automorphism. Moreover, $G$ is infinite and finitely generated, and
is therefore not locally finite. Thus the local finiteness assumption
in Theorem~\ref{thm:proper-subgroups} cannot be omitted.
\end{example}

\subsection{Finite quotients of residually finite groups}

\begin{proposition}\label{prop:residually-finite-quotients}
Let $G$ be a residually finite group. If every finite quotient of $G$
admits an inverse ambiguous automorphism, then $G$ is abelian.
\end{proposition}

\begin{proof}
Let $N\unlhd G$ be a normal subgroup of finite index. By hypothesis,
$G/N$ admits an inverse ambiguous automorphism, and hence is abelian
by Proposition~\ref{prop:automorphism-rigidity}. Therefore
\[
[G,G]\leq N.
\]
This holds for every normal subgroup $N$ of finite index. Since $G$
is residually finite,
\[
\bigcap_{\substack{N\unlhd G\\ [G:N]<\infty}}N=1.
\]
Consequently,
\[
[G,G]=1,
\]
and $G$ is abelian.
\end{proof}

The conclusion of
Proposition~\ref{prop:residually-finite-quotients} cannot, in
general, be strengthened to the existence of an inverse ambiguous
automorphism of $G$.

\begin{theorem}\label{thm:finite-quotients-no-lifting}
There exists an infinite residually finite abelian group such that every
finite quotient admits an inverse ambiguous automorphism, whereas the
group itself admits none.
\end{theorem}

\begin{proof}
Let $S$ be the multiplicative subset of $\mathbb Z$ consisting of the
positive integers all of whose prime divisors are either equal to $2$
or congruent to $3$ modulo $4$. Regard
\[
A=S^{-1}\mathbb Z
\]
as an additive group.

We first prove that $A$ is residually finite. Let
\[
0\neq \frac{m}{s}\in A,
\]
where $m\in\mathbb Z$ and $s\in S$. Choose a prime divisor $p$ of
\[
4m^2+1.
\]
Then $p$ is odd and
\[
(2m)^2\equiv -1\pmod p.
\]
Thus the residue class of $2m$ has order $4$ in
$(\mathbb Z/p\mathbb Z)^\times$, and consequently
\[
p\equiv 1\pmod 4.
\]
In particular, $p$ is relatively prime to every element of $S$.
Reduction modulo $p$ therefore extends to a homomorphism
\[
\rho_p\colon A\longrightarrow\mathbb Z/p\mathbb Z,
\qquad
\rho_p\left(\frac{a}{t}\right)=at^{-1}\pmod p.
\]
Since $p\nmid m$, we have
\[
\rho_p\left(\frac{m}{s}\right)\neq 0.
\]
Hence $A$ is residually finite.

Let
\[
F=A/K
\]
be a finite quotient, and put
\[
c=1+K.
\]
For each $s\in S$, multiplication by $s$ induces an endomorphism
\[
\mu_s\colon F\longrightarrow F,
\qquad
\mu_s(x+K)=sx+K.
\]
Indeed, $sK\subseteq K$. Multiplication by $s$ is surjective on $A$,
because
\[
\frac{a}{t}
 =
s\frac{a}{st}.
\]
Therefore $\mu_s$ is surjective. Since $F$ is finite, $\mu_s$ is an
automorphism.

We claim that $F=\langle c\rangle$. Let
\[
\overline{x}=\frac{a}{s}+K\in F.
\]
Then
\[
\mu_s(\overline{x})=a+K=ac.
\]
The subgroup $\langle c\rangle$ is invariant under $\mu_s$. Since the
restriction of $\mu_s$ to the finite group $\langle c\rangle$ is
injective, it is surjective. Hence there exists
$y\in\langle c\rangle$ such that
\[
\mu_s(y)=ac.
\]
Since $\mu_s$ is injective on $F$, it follows that
\[
\overline{x}=y\in\langle c\rangle.
\]
Thus $F$ is cyclic.

If $F$ is trivial, it admits the identity automorphism. Suppose that
$F$ is non-trivial, and write
\[
F\cong C_n.
\]
We claim that no prime in $S$ divides $n$. Indeed, suppose that a
prime $q\in S$ divides $n$. Since $c$ generates $F$,
\[
z=\frac{n}{q}c
\]
is non-zero and satisfies
\[
qz=0.
\]
This contradicts the fact that multiplication by $q$ is an
automorphism of $F$.

Consequently, every prime divisor of $n$ is congruent to $1$ modulo
$4$. It follows from Lemma~\ref{lem:torsion-primary} that $F$ admits
an inverse ambiguous automorphism. Thus every finite quotient of $A$
admits an inverse ambiguous automorphism.

It remains to prove that $A$ itself admits no such automorphism. Let
\[
\varphi\colon A\longrightarrow A
\]
be an additive endomorphism, and set
\[
r=\varphi(1)\in A.
\]
For
\[
x=\frac{a}{s}\in A,
\]
we have
\[
s\varphi(x)
 =\varphi(sx)
 =\varphi(a)
 =ar.
\]
On the other hand,
\[
s(rx)=r(sx)=ar.
\]
Since $A$ is torsion-free, multiplication by $s$ is injective.
Therefore
\[
\varphi(x)=rx.
\]
Thus every endomorphism of the additive group $A$ is multiplication
by an element of $A\subseteq\mathbb Q$.

If $\varphi$ were an inverse ambiguous automorphism, then, in
additive notation,
\[
\varphi^2=-\id_A.
\]
Evaluating this identity at $1$ gives
\[
r^2=-1,
\]
which is impossible in $\mathbb Q$. Hence $A$ admits no inverse
ambiguous automorphism.
\end{proof}


\begin{thebibliography}{99}

\bibitem{ChengEtAl}
R.~Cheng, A.~Dasgupta, B.~R. Ebanks, L.~F. Kinch, L.~M. Larson,
and R.~B. McFadden,
``When does $f^{-1}=1/f$?'',
\emph{Amer. Math. Monthly} \textbf{105} (1998), no.~8, 704--717.
\doi{10.2307/2588987}

\bibitem{EulerForan}
R.~Euler and J.~Foran,
``On functions whose inverse is their reciprocal'',
\emph{Math. Mag.} \textbf{54} (1981), no.~4, 185--189.

\bibitem{Griffiths}
M.~Griffiths,
``$f(f(x))=-x$, windmills, and beyond'',
\emph{Math. Mag.} \textbf{83} (2010), no.~1, 15--23.

\bibitem{OlshanskiiTarski}
A.~Yu. Ol'shanskii,
``Groups of bounded period with subgroups of prime order'',
\emph{Algebra Logic} \textbf{21} (1982), no.~5, 369--418.
\doi{10.1007/BF02027230}

\bibitem{Schmitz}
D.~J. Schmitz,
``Inverse ambiguous functions on fields'',
\emph{Aequationes Math.} \textbf{91} (2017), no.~2, 373--389.
\doi{10.1007/s00010-016-0464-5}

\bibitem{SchmitzGallagher}
D.~J. Schmitz and K.~Gallagher,
``Inverse ambiguous functions on some finite non-abelian groups'',
\emph{Aequationes Math.} \textbf{92} (2018), no.~5, 963--973.
\doi{10.1007/s00010-018-0542-y}

\bibitem{schmitzrahmankindness2025continuous}
D.~J. Schmitz, S.~Rahman, and A.~Kindness,
``Continuous inverse ambiguous functions on Lie groups'',
\emph{Aequationes Math.} \textbf{99} (2025), 1357--1369.
\doi{10.1007/s00010-024-01131-8}

\bibitem{Toborg}
I.~Toborg,
``Inverse ambiguous functions and automorphisms on finite groups'',
\emph{Ann. Math. Siles.} \textbf{33} (2019), 284--297.
\doi{10.2478/amsil-2019-0006}

\end{thebibliography}
\end{document}